\documentclass[12pt,reqno]{amsart}

\usepackage{amsmath}
\usepackage{comment}
\usepackage{amsfonts}
\usepackage{amssymb}

\newtheorem{proposition}{Proposition}[section]
\newtheorem{theorem}[proposition]{Theorem}
\newtheorem{corollary}[proposition]{Corollary}
\newtheorem{lemma}[proposition]{Lemma}

\theoremstyle{definition}
\newtheorem{definition}[proposition]{Definition}

\newtheorem{remark}[proposition]{Remark}

\numberwithin{equation}{section}

\def\Dx{\Delta_x}
\def\({\left(}
\def\){\right)}
\def\Nx{\nabla_x}
\def\Dt{\partial_t}
\def\divv{\operatorname{div}}
\newcommand{\rot}{\mathop\mathrm{curl}}

\begin{document}

\title[Damped Navier-Stokes equations in $\mathbb R^2$]{Upper bounds for the attractor dimension of damped Navier-Stokes equations in $\mathbb R^2$}
\author[A. Ilyin, K. Patni and S. Zelik]{Alexei Ilyin${}^1$,
Kavita Patni ${}^2$ and Sergey Zelik${}^{1,2}$}

\begin{abstract}
We consider finite energy solutions for the damped and driven two-dimensional
 Navier--Stokes equations in the plane and show that
 the corresponding dynamical system possesses a global attractor.
 We obtain upper bounds for its fractal dimension when the forcing term
 belongs to  the
whole scale of homogeneous Sobolev spaces from $-1$ to $1$.
\end{abstract}

\subjclass[2000]{35B40, 35B45}
\keywords{Damped Navier-Stokes equations, attractors, unbounded domains, box-counting dimension}
\address{\newline
${}^1$ Keldysh Institute of Applied Mathematics, \newline
 Moscow 125047, Russia\newline
\ \ ${}^2$ University of Surrey, Department of Mathematics, \newline
Guildford, GU2 7XH, United Kingdom.}
\email{ilyin@keldysh.ru}
\email{k.patni@surrey.ac.uk}
\email{s.zelik@surrey.ac.uk}

\maketitle

\section{Introduction}\label{s1}

The theory of global attractors for the 2-D Navier--Stokes
system
$$
\begin{cases}
\Dt u+(u,\Nx)u+\Nx p=\nu\Dx u+g,\\
u\big|_{t=0}=u_0,\ \ \divv u=0.
\end{cases}
$$
has been a starting point of the theory of infinite
dimensional dissipative dynamical systems and remains in
the focus of this theory, see~\cite{BV, Ch-V-book,CF88, FMRT,Lad,temam,S-Y}
and the references therein.

In the case of a bounded domain $\Omega$  the corresponding dynamical system
has a global attractor in the appropriate phase space. The attractor
has finite fractal dimension, measured in terms of the  dimensionless
number $G$ (the Grashof number), $G=\frac{\|g\||\Omega|}{\nu^2}$.

The best known estimate in the case of the Dirichlet boundary conditions
$u\vert_{\partial\Omega}=0$ is (see~\cite{temam})
\begin{equation}\label{0.est-dir}
\dim_f\mathcal{A}\le c_{\mathrm{D}}G,
\end{equation}
while in the case of a periodic domain $x\in[0,2\pi L]^2$
the estimate can be significantly improved (see \cite{CFT}):
$$
\dim_f\mathcal{A}\le c_{\mathrm{per}}G^{2/3}(\ln(1+G))^{1/3},
$$
and, moreover, this estimate  sharp up to a logarithmic correction as
shown in \cite{Liu}, see also \cite{DG} for the alternative proof of the upper bound.

Finding explicit majorants for the constants $c_{\mathrm{D}}$ and
$c_{\mathrm{per}}$ amounts to finding sharp or explicit constants
in certain Sobolev inequalities and spectral Lieb--Thirring
inequalities. For example, $c_{\mathrm{D}}\le (4\pi 3^{1/4})^{-1}$
\cite{I-Stokes}, and the majorant for $c_{\mathrm{per}}$ can
be easily written down using the recent result \cite{BDZ} on the
  sharp constant in the logarithmic
Brezis--Gallouet inequality (which is essential for the attractor
dimension estimate in the periodic case).

In unbounded channel-like domains the Navier-Stokes system with Di\-richlet boundary conditions
is still dissipative in view of the Poincar\'e inequality. In particular, if the case of finite 
energy solutions is considered, the associated semigroup possesses a compact global attractor 
similarly to the case of bounded domains. Up to the moment, there are two alternative ways to 
establish this fact. The first one is based on the weighted energy estimates and careful analysis 
of the Leray projection in weighted Sobolev spaces, see \cite{Bab,MZ} and references therein 
and the second one utilizes the so-called energy method (which will be also used in our paper) 
and the energy equality, see \cite{ball,Rosa,temam}.
\par
 However, in contrast to the case of bounded domains,
the solution semigroup is no longer compact, but is  rather
asymptotically compact and this fact strongly affects the existing upper bounds for 
the dimension of the attractor. Indeed, the best known estimate for the case of 
channel-like domains
obtained in \cite{Rosa} can be written as follows
$$
\dim_f\mathcal A\le \frac{c_\mathrm{LT}}2\frac{\|g\|^2} {\lambda_1^{2}\nu^{4}}.
$$
Here $\lambda_1>0$ is the bottom of the spectrum of the Stokes operator in the channel,
and ${c_\mathrm{LT}}$ is the universal  Lieb--Thirring constant, see \eqref{L-T}.
For example, in a straight channel of width $d$ we have $\lambda_1\ge \pi^2/d^2$,
so that in this case we obtain
$$
\dim_f\mathcal A\le \frac{1}{4\sqrt{3}\pi^4}\frac{d^4\|g\|^2} {\nu^{4}}.
$$
We observe that these estimates are proportional to $\nu^{-4}$ (unlike
 \eqref{0.est-dir} which is proportional to $\nu^{-2}$). On the other hand, it is worth 
 mentioning that, to the best of our knowledge, no growing as $\nu\to0$ lower bounds for the  
 dimension of the attractor are known for the case of Dirichlet boundary conditions, 
 regardless whether the underlying domain is bounded or unbounded. So, in contrast to 
 the case of periodic boundary conditions,  the behaviour of the attractor's dimension 
 as $\nu\to0$ remains  unclear for the case of Dirichlet boundary conditions even in 
 the case of bounded domains.

We mention also that keeping in mind the Poiseulle flows and the mean flux integral, 
it seems more natural to consider the {\it infinite} energy solutions for the 
Navier-Stokes system in a pipe. In this case, the system remains dissipative 
and the existence of the so-called locally compact attractor can be established, 
see \cite{AZ, Zel-glas}. The dimension of this attractor may be infinite in 
general, but will be finite if the external forces $g(x)$ decay to zero as 
$|x|\to\infty$ (no matter how slow this decay is), see \cite{MZ,Zel-glas} and the
 references therein.

In the whole $\mathbb{R}^2$ the Laplacian is not positive-definite
and the Navier--Stokes system is not dissipative at least in a usual sense 
even in the case of zero external forces and finite energy solutions, 
see e.g. \cite{shonbek}, see also \cite{Gal1,Wie}  and the references 
therein concerning the decay properties of various types of solutions 
for the Navier-Stokes problem with zero external forces. The presence 
of external forces makes the problem more complicated and usually only 
growing in time bounds for the solutions are available. We mention here 
only the recent results concerning the polynomial growth in time for 
the so-called uniformly local norms of infinite-energy solutions obtained 
in \cite{Gal, Zel-inf}, see also the references therein.
\par
 Let us now consider the  damped and driven Navier--Stokes system
\begin{equation}\label{1.ns-main}
\begin{cases}
\Dt u+(u,\Nx)u+\Nx p+\alpha u=\nu\Dx u+g,\\
u\big|_{t=0}=u_0,\ \ \divv u=0.
\end{cases}
\end{equation}
with additional dissipative term $\alpha u$.
The drag/friction term $\alpha u$, where $\alpha>0$
is the Rayleigh or  Ekman friction coefficient (or the Ekman
 pumping/dissipation constant),
 models the bottom friction in
two-dimensional oceanic models
and is the main energy sink in large scale atmospheric models~\cite{P}.
\par
The analytic properties of system \eqref{1.ns-main} (such as existence of solutions, their uniqueness and regularity, etc.) remain very close to the analogous properties of the classical Navier-Stokes equations or Euler equations if the inviscid case $\nu=0$ is considered. However,  the friction term
$\alpha u$ is very essential for the dynamics since it removes the energy which piles up at the large scales and from the mathematical point of view compensates the lack of the Poincare inequality. This
 makes the Navier--Stokes system and even the limit Euler system dissipative whatever
the domain is and allows to study its global attractors in various phase spaces. For instance, the so-called weak global attractor for the inviscid case is constructed in~\cite{I91} for the case of finite energy solutions; its compactness in the strong topology related with the $H^1$-norm is verified in \cite{CVZ} and \cite{CZ} for the cases of finite and infinite energy solutions respectively; the inviscid limit $\nu\to0$ is studied in \cite{const} including the absence of the so-called anomalous dissipation of enstrophy; the existence of a locally compact global attractor in the uniformly local phase spaces is established for the viscous case $\nu>0$ in \cite{Zel-inf}; see also~\cite{BCT}, \cite{W}
for the existence and uniqueness results for the stationary problem
and the
stability of stationary solutions for~\eqref{1.ns-main}
with~$\nu=0$.

From the point of view of  the attractors
and their dimension the
system~\eqref{1.ns-main}
in the case of the  periodic domain $x\in[0,2\pi L]^2$ was studied
in \cite{I-M-T},
where it was shown that the corresponding dynamical system
possesses a global attractor $\mathcal{A}$ (in $L^2$) whose fractal dimension
is finite and satisfies the following estimate
\begin{equation}\label{min}
\dim_f\mathcal{A}\le \min
\left(\sqrt{6}\,
\frac{\|\rot g\|L}{\nu\alpha}\,,\
\frac38\,\frac{\|\rot g\|^2}{\nu\alpha^3}\right),
\end{equation}
where
the values of the constants are updated in accordance with
\cite{I12JST}. The first estimate
is enforced in the regime $\nu/L^2\gg\alpha$, while the second
 estimate
is enforced in the opposite regime $\nu/L^2\ll\alpha$. We observe that both
estimates are of the order $1/\nu$ as $\nu\to0^+$  if all the remaining parameters
are fixed. It was also shown in \cite{I-M-T} that this rate
of growth of the dimension is sharp, and the upper bounds were
supplemented with a lower bound of the order $1/\nu$, based on
the instability analysis of generalized Kolmogorov flows. The finite-dimensionality of the global attractor in the uniformly local phase spaces under the assumption that the external forces $g(x)$ decay to zero as $|x|\to\infty$ has been proved recently in \cite{Pen}, but no explicit upper bounds for the dimension was given there.

We would also like to point out that starting from the
paper~\cite{Lieb} the Lieb--Thirring inequalities are
an essential analytical tool  in the estimates
of global  Lyapynov exponents for the Navier--Stokes equations.
This fully applies to our case.

We now observe that the first estimate in~\eqref{min} blows up as
the size of the periodic domain $L\to \infty$. On the other hand,
the second estimate survives (the homogeneous $H^1$-norm is
scale invariant in two dimensions). Therefore, one might expect
that this estimate holds  for $L=\infty$, that is, for
$x\in\mathbb{R}^2$, and a motivation of the present work
is to show that this is indeed the case.

In this paper we study the damped and driven Navier--Stokes system
\eqref{1.ns-main} in $\mathbb{R}^2$ in the class of {\it finite} energy
solutions and our main aim is to obtain  explicit upper bounds for the attractor's dimension in terms of the parameters $\nu$ and $\alpha$ and various norms of the external forces $g$.
\par
In section~\ref{s2} we recall for the reader convenience the proof
of the well-posedness and derive the energy equality. Then
using the energy equality method \cite{rosa,Rosa} we establish the
asymptotic compactness of the solution semigroup and, hence,
the existence of the global attractor~$\mathcal{A}$.

In section~\ref{s3} we consider the case when the right-hand
side $g$ belongs to the scale of homogeneous Sobolev spaces
$\dot H^s$, $s\in[-1,1]$ and derive the following estimate(s)
for the fractal dimension of the attractor $\mathcal{A}$:
$$
\dim_f\mathcal A\le \frac{1-s^2}{64\sqrt{3}}
\(\frac{1+|s|}{1-|s|}\)^{|s|}\frac1{\alpha^{2+s}\,\nu^{2-s}}\|g\|^2_{\dot H^{s}},
\quad s\in[-1,1].
$$
In particular, for $s=1$ we obtain
$$
\dim_f\mathcal A\le \frac{1}{16\sqrt{3}}
\frac{\|\rot g\|^2}{\nu\,\alpha^{3}},
$$
which up to a constant coincides with the second estimate
in \eqref{min}, proving thereby our expectation.

In this paper we  use standard notation.
The $L_2$-norm and the corresponding scalar product
are denoted by $\|\cdot\|$ and $(\cdot,\cdot)$.

\section{A priori estimates, well-posedness and
asymptotic compactness}\label{s2}
We study the damped and driven  Navier-Stokes system \eqref{1.ns-main}
in $\mathbb R^2$.
Here $u(t,x)=(u^1,u^2)$ is the unknown velocity vector field, $p$ is
the unknown pressure, $g(x)=(g^1,g^2)$ is the given external force
(and without loss of generality we can and shall assume that
$\divv g=0$),
the advection term is
$$
(u,\Nx)v=\sum_{i=1}^2u^i\partial_{x_i}v,
$$
 and
$\alpha>0$, $\nu>0$ are given parameters. We restrict ourselves
to considering only {\it finite} energy solutions, so we assume that
$$
g\in [L^2(\mathbb R^2)]^2,\ \ u_0\in \mathcal H:=\{u_0\in[L^2(\mathbb R^2)]^2,\  \divv u_0=0\},
$$
and by definition $u=u(t,x)$ is a weak solution of
\eqref{1.ns-main} if
\begin{equation}\label{1.def}
u\in C(0,T;\mathcal H)\cap L^2(0,T; [H^1(\mathbb R^2)]^2),\ \ T>0,
\end{equation}
and \eqref{1.ns-main} is satisfied in the sense of distributions.
This means that  for every divergence free test function
$\varphi(t,x)=(\varphi^1,\varphi^2)$,
$\divv \varphi=0$, $\varphi\in C_0^\infty(R_+\times\mathbb R^2)$,
the following integral identity holds
\begin{multline}\label{1.weak}
-\int_{\mathbb R}(u,\Dt\phi)\,dt+\int_{\mathbb R}((u,\Nx)u,\varphi)\,dt+\alpha\int_{\mathbb R}(u,\varphi)\,dt+\\
+\nu\int_{\mathbb R}(\Nx u,\Nx\varphi)\,dt=\int_{\mathbb R}(g,\varphi)\,dt.
\end{multline}

The following fact concerning the global well-posedness of
the Navier-Stokes equations in 2D is well-known, see, for instance,
 \cite{temam,temam1}.

\begin{theorem}\label{Th1.main}
For any $u_0\in\mathcal H$ there exists a unique
solution $u$ of problem \eqref{1.ns-main} satisfying
\eqref{1.def} and \eqref{1.weak}. Moreover, this solution
satisfies the following energy equality for almost all $t\ge0$:
\begin{equation}\label{1.energy}
\frac12 \frac d{dt}\|u(t)\|^2+\nu\|\Nx u(t)\|^2+
\alpha\|u(t)\|^2=(g,u(t)).
\end{equation}
\end{theorem}
\begin{proof} For the convenience of the reader, we remind below the key steps of the proof. 
It is strongly based on the so-called Ladyzhenskaya interpolation inequality
\begin{equation}\label{1.lad}
\|u\|_{L^4(\mathbb R^2)}^2\le C\|u\|\|\Nx u\|
\end{equation}
which holds for any $u\in H^1(\mathbb R^2)$. Indeed, this inequality together with \eqref{1.def} 
implies that any weak solution $u$ satisfies
$$
u\in L^4(0,T;L^4(\mathbb R^2))
$$
and, integrating by parts and using the Cauchy-Schwartz inequality, we see that
$$
|(u,\Nx)u,\varphi)|=
\bigl|\sum_{i,j=1}^2(u^iu^j,\partial_{x_j}\varphi^i)\bigr|\le C\|u\|^2_{L^4}\|\varphi\|_{H^1}.
$$
Therefore,
$$
\int_{\mathbb R}((u,\Nx u),\varphi)\,dt\le C\|u\|^2_{L^4(0,T;L^4(\mathbb R^2))}\|\varphi\|_{L^2(0,T;H^1(\mathbb R^2))}
$$
and
\begin{equation}\label{1.neg}
(u,\Nx)u\in L^2(0,T;\mathcal H^{-1})=[L^2(0,T;\mathcal H^1)]^*,
\end{equation}
where, as usual, $\mathcal H^1$ is a subspace of $[H^1(\mathbb R^2)]^2$
which consists of divergence free vector fields and
$\mathcal H^{-1}:=[\mathcal H^1]^*$. Thus, from \eqref{1.weak},
we see that $\Dt u\in L^2(0,T;\mathcal H^{-1})$ as well, and all
terms in \eqref{1.weak} make sense for the test function
$\varphi=u\in L^2(0,T;\mathcal H^1)$, so by approximation
arguments (and the fact that the divergent free vector fields
with compact support are dense in $\mathcal H$) we may take
$\varphi=u$ in \eqref{1.weak}. Then, using the well-known
orthogonality relation
\begin{equation}\label{1.null}
((u,\Nx)v,v)=-\frac12(\divv u,|v|^2)=0,
\end{equation}
we end up with the desired energy equality \eqref{1.energy},
see, for instance, \cite{temam} for the details.
\par
The existence of a weak solution can be obtained in a standard way
based on the energy equality \eqref{1.energy} either directly by
the Faedo-Galerkin method or by approximation the problem
\eqref{1.ns-main} in $\mathbb R^2$ by the corresponding problems
in bounded domains (see e.g., \cite{temam,temam1} for the details),
 so we leave this proof to the reader and remind below the proof
 of the uniqueness. Let $u_1$ and $u_2$ be two weak solutions of
 problem \eqref{1.ns-main} and let $v=u_1-u_2$. Then $v$ solves
$$
\Dt v+(u_1,\Nx)v+(v,\Nx)u_2+\alpha v+\Nx \bar p=
\nu\Dx v,\ \ \divv v=0.
$$
Multiplying this equation by $v$ and integrating over
$x\in\mathbb R^2$, analogously to the energy equality, we have
\begin{equation}\label{1.dif-en}
\frac12\frac d{dt}\|v(t)\|^2+\alpha\|v(t)\|^2+
\nu\|\Nx v(t)\|^2=-((u_1,\Nx v),v)-((v,\Nx)u_1,v).
\end{equation}
The first term on the right-hand side of this equality vanishes
in view of~\eqref{1.null} and the second
term can be estimated using the Ladyzhenskaya inequality as follows:
\begin{multline*}
|((v,\Nx)u_2,v)|\le\|v\|^2_{L^4}\|\Nx u_2\|\le
 C\|v\|\|\Nx v\|\|\Nx u_2\|\le\\\le \nu\|\Nx v\|^2+
 C^2\nu^{-1}\|\Nx  u_2\|^2\|v\|^2.
\end{multline*}
Inserting this estimate into the right-hand side of
\eqref{1.dif-en} we obtain
$$
\frac12\frac d{dt}\|v(t)\|^2\le
 C^2\nu^{-1}\|\Nx u_2(t)\|^2\|v(t)\|^2,
$$
and the Gronwall inequality  gives
$$
\|u_1(t)-u_2(t)\|^2\le
 e^{2C^2\nu^{-1}\int_0^t\|\Nx u_2(s)\|^2ds}\|u_1(0)-u_2(0)\|^2.
$$
Thus, the uniqueness is proved and the theorem is also proved.
\end{proof}
\begin{corollary}\label{Cor1.dis} The weak solution $u(t)$ of problem \eqref{1.ns-main}
satisfies the following dissipative estimate:
\begin{equation}\label{1.dis}
\aligned
\|u(t)\|^2&\le\|u_0\|^2e^{-\alpha t}+\alpha^{-2}\|g\|^2,\\
2\nu\int_t^{t+T}\|\Nx u(s)\|^2ds&\le\alpha^{-1}T\|g\|^2+\|u(t)\|^2.
\endaligned
\end{equation}
\end{corollary}
\begin{proof}
Using the Cauchy--Schwartz
inequality,  the first
estimate follows from~\eqref{1.energy} by the Gronwall inequality,
and the second estimate is proved by integrating~\eqref{1.energy}.
\end{proof}
Thus, equation \eqref{1.ns-main} defines a solution semigroup $S(t):\mathcal H\to\mathcal H$:
$$
S(t)u_0:=u(t),\ \ u_0\in\mathcal H
$$
where $u(t)$ is a weak solution of equation \eqref{1.ns-main}
with the initial data $u(0)=u_0$. Moreover, this semigroup
is dissipative according to estimate \eqref{1.dis} and is Lipschitz continuous:
for any $u_1,u_2\in\mathcal H$
\begin{equation}\label{1.lipp}
\|S(t)u_1-S(t)u_2\|\le Ce^{Kt}\|u_1-u_2\|,
\end{equation}
where the constants $C$ and $K$ depend only on $\|u_0^i\|$.
Our next aim is to verify the existence of a global attractor for
this semigroup. For the reader convenience, we first remind its
definition, see \cite{BV,Ch-V-book, CF88,temam} for more details.

\begin{definition}\label{Def1.attr} Let $S(t)$, $t\ge0$, be a semigroup acting in a Banach space $\mathcal H$. Then the set $\mathcal A\subset\mathcal H$ is a global attractor of the semigroup $S(t)$ if
\par
1) The set $\mathcal A$ is compact in $\mathcal H$.
\par
2) It is strictly invariant: $S(t)\mathcal A=\mathcal A$.
\par
3) It attracts the images of bounded sets in $\mathcal H$ as $t\to\infty$, i.e., 
for every bounded set $B\subset \mathcal H$ and every neighborhood 
$\mathcal O(\mathcal A)$ of the set $\mathcal A$ in $\mathcal H$ 
there exists $T=T(B,\mathcal O)$ such that
$$
S(t)B\subset\mathcal O(\mathcal A)
$$
for all $t\ge T$.
\end{definition}
To state the abstract theorem on the existence of a global attractor, we need more definitions.

\begin{definition} Let $S(t)$ be a semigroup in a Banach space $\mathcal H$. 
Then, a set $\mathcal B\subset\mathcal H$ is an {\it absorbing} set of $S(t)$ if, 
for every bounded set $B\subset\mathcal H$, there exists $T=T(B)$ such that
$$
S(t)B\subset\mathcal B.
$$
A semigroup $S(t)$ is {\it asymptotically compact} if for
any bounded sequence $u_0^n$ in $\mathcal H$ and for any
sequence $t_n\to\infty$, the sequence $S(t_n)u_0^n$ is precompact in $\mathcal H$.
\end{definition}
In order to verify the existence of a global attractor we
will use the following criterion, see \cite{BV,Ch-V-book,Lad,S-Y,temam} for its proof.

\begin{proposition}\label{Prop1.attr} Let $S(t)$ be a semigroup
in a Banach space $\mathcal H$. Suppose that
\par
1) $S(t)$ possesses a bounded closed absorbing set $\mathcal B\subset H$;
\par
2) $S(t)$ is asymptotically compact;
\par
3) For every fixed $t\ge0$ the map $S(t):\mathcal B\to\mathcal H$ is continuous.
\par
Then the semigroup $S(t)$ possesses a global attractor
$\mathcal A\subset \mathcal B$. Moreover, the attractor $\mathcal A$ has the following structure:
$$
\mathcal A=\mathcal K\big|_{t=0},
$$
where $\mathcal K\subset L^\infty(\mathbb R,\mathcal H)$ is the set of complete
trajectories $u:\mathbb R\to\mathcal H$ of semigroup $S(t)$ which are defined
for all $t\in\mathbb R$ and bounded.
\end{proposition}
The next theorem which establishes the existence of a global attractor for the solution
semigroup $S(t)$ associated with equation \eqref{1.ns-main} is the main result of this section.

\begin{theorem}\label{Th1.attr}
The solution semigroup  $S(t)$ of the damped Navier-Stokes problem \eqref{1.ns-main}
possesses a global attractor $\mathcal A$ in $\mathcal H$.
\end{theorem}
\begin{proof} We will check the assumptions of Proposition
\ref{Prop1.attr}. Indeed, the first assumption is satisfied
due to the dissipative estimate \eqref{1.dis} and the desired
bounded and closed absorbing set can be taken as
\begin{equation}\label{1.abs}
\mathcal B:=\big\{u_0\in\mathcal H\,:\ \|u_0\|_{L^2}^2\le 2\alpha^{-2}\|g\|^2_{L^2}\big\}.
\end{equation}
The third assumption is also satisfied due to estimate \eqref{1.lipp}.
Thus, we only need to check the asymptotic compactness. We will use the
so-called energy method (see \cite{ball,rosa, Rosa}) in order to do this.
\par
Indeed, let $u^0_n\in\mathcal H$ be a bounded sequence of the
initial data. Then, due to estimate \eqref{1.dis}, we may assume
without loss of generality that $u^0_n\in\mathcal B$.
Let $u_n(t)$, $t\ge-t_n$, $t_n\to\infty$,  be the sequence of solutions of
\begin{multline*}
\Dt u_n+(u_n,\Nx)u_n+\alpha u_n+\Nx p_n=\nu\Dx u_n+g,\\
\ \divv u_n=0,\ \ u_n\big|_{t=-t_n}=u_n^0.
\end{multline*}
Then, $u_n(0)=S(t_n)u_n^0$ and we only need to verify that the
sequence $\{u_n(0)\}_{n=0}^\infty$ is precompact in $\mathcal H$.
In order to do so, we first verify that there exists a subsequence
 (which we also denote by $u_n$ for simplicity) such that
\begin{equation}\label{1.weak-n}
u_n(0)\rightharpoondown u(0)
\end{equation}
converges weakly in $\mathcal H$ to  $u(0)$, where
$u(t)$ is a complete bounded
trajectory $u\in \mathcal{K}$.  We first note that due to the dissipative estimate \eqref{1.dis},
\begin{equation}\label{1.uniform}
\|u_n\|_{L^\infty(-T,T;\mathcal H)}+
\sup_{t\ge -T}\|u_n\|_{L^2(t,t+1;\mathcal H^1)}\le C,
\end{equation}
where $T\le t_n$ and $C$ is independent of $n$ and $T$. Moreover,
from the Ladyzhenskaya inequality \eqref{1.lad} we conclude also
that $u_n$ is bounded in $L^4(-T,T;L^4)$. Thus, by the
Banach-Alaoglu theorem, we may assume without loss of generality that
\begin{multline*}
u^n\to u\ \text{weak-star in }\ L^\infty(-T,T;\mathcal H) \\
\text{and weakly in\ } L^2(-T,T;H^1)\cap L^4(-T,T;L^4)
\end{multline*}
for every $T\in\mathbb N$ to some function $u\in L^\infty(\mathbb R;\mathcal H)\cap L^2_{loc}(\mathbb R,H^1)$ which also satisfies estimate \eqref{1.uniform}. Moreover, analogously to \eqref{1.neg}, we conclude that $\Dt u_n$ is uniformly bounded in $L^2(t,t+1;\mathcal H^{-1})$. Thus, without loss of generality
$$
\Dt u^n\rightharpoondown\Dt u
$$
in the space $L^2(t,t+1;\mathcal H^{-1})$ for all $t\in\mathbb R$. Using
the embedding
$$
L^2(t,t+1;\mathcal H^{-1})\cap L^2(t,t+1;\mathcal H^1)\subset C(t,t+1;\mathcal H),
$$
see  \cite[Lemma III.1.2]{temam},
 we see that $u_n(0)\rightharpoondown u(0)$. Thus, to verify
 \eqref{1.weak-n}, we only need to check that $u(t)$, $t\in\mathbb R$ is
 a weak solution of \eqref{1.ns-main}. In other words, we need to
 pass to the limit as $n\to\infty$ in the analogue of \eqref{1.weak}:
\begin{multline}\label{1.weak1}
-\int_{\mathbb R}(u_n,\Dt\phi)\,dt+\int_{\mathbb R}((u_n,\Nx)u_n,\varphi)\,dt+
\alpha\int_{\mathbb R}(u_n,\varphi)\,dt+\\+\nu\int_{\mathbb R}
(\Nx u_n,\Nx\varphi)\,dt=\int_{\mathbb R}(g,\varphi)\,dt,
\end{multline}
where $\varphi(t,x)\in C_0^\infty(\mathbb R^3)^2$ is an arbitrary fixed
divergence free function. Since passing to the limit in the linear
terms is evident, we only need to pass to the limit $n\to\infty$ in
the non-linear term. Integrating by parts and rewriting the nonlinear term in the form
$$
((u_n(t),\Nx)u_n(t),\varphi(t))=-\sum_{i,j=1}^2
(u_n^i(t)u_n^j(t)\partial_{x_i}\varphi^j(t))
$$
and using the fact that the support of $\varphi$ is finite, we conclude that, for passing to the limit in the nonlinear term, it is enough to verify that, for every fixed $R,T>0$,
\begin{equation}\label{1.strong}
u_n\to u \ \text{ strongly in } L^2(-T,T;L^2(B^R_0)),
\end{equation}
where $B^R_0$ stands for the ball of radius $R$ in $\mathbb R^2$ centered
at the origin. To verify the convergence \eqref{1.strong}, we
recall that the sequence $u_n$ is bounded in the space
$L^2(-T,T;\mathcal H^1(B^R_0))$ due to the dissipative estimate
\eqref{1.dis}. Moreover, analogously to \eqref{1.neg} but using
the test functions $\varphi\in L^2(-T,T;\mathcal H^1_0(B^R_0))$,
we see that $\Dt u_n$ is bounded in the space
$L^2(-T,T;\mathcal H^{-1}(B^R_0))$. Thus, since
$$
\mathcal H^{1}(B^R_0))\subset\mathcal H(B^R_0)\subset\mathcal H^{-1}(B^R_0)
$$
and the first embedding is compact, the
compactness theorem (see, for instance~\cite[Theorem~III.2.1]{temam})
 implies that the embedding
$$
H^1(-T,T;\mathcal H^{-1}(B^R_0))\cap L^2(-T,T;\mathcal H^1(B^R_0))\subset L^2(-T,T;\mathcal H)
$$
is compact. Therefore, $u_n$ is precompact in
$L^2(-T,T;L^2(B^R_0))$ for every $R>0$, $T>0$ and passing to
a subsequence if necessary, we conclude that the convergence
\eqref{1.strong} indeed holds. Thus, passing to the limit in the
nonlinear term of \eqref{1.weak1} is verified and $u$ is a weak
solution of \eqref{1.ns-main} which is defined for all $t\in\mathbb{R}$ and bounded, so
$u\in\mathcal K$. This means that the
convergence \eqref{1.weak-n} is verified.
\par
We are now ready to verify that
\begin{equation}\label{1.strong-n}
u_n(0)\to u(0) \ \text{strongly in } \ \mathcal H
\end{equation}
and finish the proof of the theorem. We multiply   the energy equality~\eqref{1.energy}
for the solutions $u_n$ by $ e^{2\alpha t}$  and integrate from $-t_n$ to $0$:
\begin{multline}\label{1.energy-n}
\|u_n(0)\|^2=-2\int_{-t_n}^0e^{2\alpha s}\|\Nx u_n(s)\|^2\,ds+\\
+\|u_n(-t_n)\|^2e^{-2\alpha t_n}+2\int_{-t_n}^0e^{2\alpha s}(g,u_n(s))\,ds.
\end{multline}
We want to pass to the limit $n\to\infty$ in this equality.
Indeed, using the weak convergence $u_n\to u$ in $L^2_{loc}(\mathbb R,H^1)$ implying that
\begin{multline*}
\limsup_{n\to\infty}-2\int_{-t_n}^0e^{2\alpha s}\|\Nx u_n(s)\|^2\,ds=\\=
-2\liminf_{n\to\infty}\int_{-t_n}^0e^{2\alpha s}\|\Nx u_n(s)\|^2\,ds\le
-2\int_{-\infty}^0e^{2\alpha s}\|\Nx u(s)\|^2\,ds
\end{multline*}
and the uniform bounds \eqref{1.uniform},  we see from~\eqref{1.energy-n} that
$$
\limsup_{n\to\infty}\|u_n(0)\|^2\le-2\int_{-\infty}^0e^{2\alpha s}\|\Nx u(s)\|^2\,ds
+ 2\int_{-\infty}^0e^{2\alpha s}(g,u(s))\,ds.
$$
On the other hand, thanks to the energy equality, for the whole-line $L^2$-bounded solution
$u\in\mathcal K$ we have
$$
\|u(0)\|^2=-2\int_{-\infty}^0e^{2\alpha s}\|\Nx u(s)\|^2\,ds+
2\int_{-\infty}^0e^{2\alpha s}(g,u(s))\,ds
$$
and, therefore, taking into the account the weak convergence
\eqref{1.weak-n}, we finally arrive at
$$
\limsup_{n\to\infty}\|u_n(0)\|^2\le\|u(0)\|^2\le\liminf_{n\to\infty}\|u_n(0)\|^2.
$$
 Thus, $\lim_{n\to\infty}\|u_n(0)\|=\|u(0)\|$, and the strong
 convergence \eqref{1.strong-n} is proved:
$$
\lim_{n\to\infty}\|u_n(0)-u(0)\|^2=
\lim_{n\to\infty}\left(\|u_n(0)\|^2+\|u(0)\|^2-2(u_n(0),u(0))\right)
=0,
$$
 and the theorem is also proved.
\end{proof}
To conclude the section, we also discuss the extra regularity of the attractor $\mathcal A$. We say that $u:\mathbb R\to\mathcal H^1$ is a {\it strong} solution of the damped Navier-Stokes equations \eqref{1.ns-main} if
$$
u\in C(0,T;\mathcal H^1)\cap L^2(0,T;H^2(\mathbb R^2))
$$
and equation \eqref{1.ns-main} is satisfied in the sense of distributions. The following analogue of Theorem \ref{Th1.main} gives the global well-posedness of strong solutions of the Navier-Stokes problem.

\begin{theorem}\label{Th1.strong}
Let  $u_0\in\mathcal H^1$. Then, the weak solution $u$ constructed
in Theorem \ref{Th1.main} is actually a strong solution and
satisfies for almost all $t\ge0$ the following analogue of the
energy equality:
\begin{equation}\label{1.strong-en}
\frac12\frac d{dt}\|\Nx u(t)\|^2+\nu\|\Dx u(t)\|^2
+\alpha\|\Nx u(t)\|^2=(g,\Dx u(t)).
\end{equation}
\end{theorem}
The proof of this theorem is analogous to Theorem \ref{Th1.main} and even
simpler since the solution $u$ is a priori more regular now, see, for instance,
\cite{temam} for more details. We only mention here that the identity
\eqref{1.strong-en} follows by multiplication of equation \eqref{1.ns-main}
by $\Dx u$, integration over $x\in\mathbb R^2$ and using the well-known
orthogonality relation
$
((u,\Nx)u,\Dx u)=0
$
which holds for any  $u\in [H^2(\mathbb R^2)]^2$, $\divv u=0$. In fact, integrating
by parts and setting $\omega:=\rot u=\partial_{x_1}u^2-\partial_{x_2}u^1$, the divergence theorem gives
$$
((u,\Nx)u,\Dx u)=\frac12\int_{\mathbb{R}^2}\divv (u \omega^2)\,dx=0
$$
for a smooth $u\in C^\infty_0(\mathbb{R}^2)^2$, $\divv u=0$. The general case follows
by a standard approximation procedure.

The next corollary gives the dissipative estimate for the solutions
of \eqref{1.ns-main} in $\mathcal H^1$, and its proof is similar to
that of~\eqref{1.dis}.
\begin{corollary}\label{Cor1.strong-dis}
The strong solution
satisfies the following estimate:
\begin{equation}\label{1.strong-dis}
\aligned
\|\Nx u(t)\|^2&\le\|\Nx u_0\|^2e^{-\alpha t}+
(2\alpha\nu)^{-1}\|g\|^2,\\
\nu\int_t^{t+T}\|\Dx u(s)\|^2ds&\le\nu^{-1}T\|g\|^2+\|\Nx u(t)\|^2.
\endaligned
\end{equation}
\end{corollary}
In the following corollary we establish the smoothing
property for the weak solutions of equation \eqref{1.ns-main}.
\begin{corollary}\label{Cor1.smo}
For  $t>0$, $u(t)\in\mathcal H^1$ and the following estimate holds:
\begin{equation}\label{1.smo}
\|u(t)\|^2_{H^1}\le Ct^{-1}\(\|u(0)\|^2+\|g\|^2\),\ 0<t\le1,
\end{equation}
where  $C=C(\alpha,\nu)$.
In particular, the attractor $\mathcal A$ is  bounded in $\mathcal H^1$.
\end{corollary}
\begin{proof}
Since $\|\Nx u(t)\|^2$ is integrable,
 for every $0<T\le1$, there exists $\tau\le T$ such
 that $u(\tau)\in\mathcal H^1$ and
$$
\|\Nx u(\tau)\|^2\le \frac1T\int_0^T\|\Nx u(s)\|^2ds\le
\frac1{2\nu T}\(\|u(0)\|^2+\alpha^{-1}T\|g\|^2\),
$$
where the second inequality follows from~\eqref{1.dis}.
By \eqref{1.strong-dis} with the initial time $t=\tau$, we have
$$
\|\Nx u(T)\|^2\le \|\Nx u(\tau)\|^2+
(2\alpha\nu)^{-1}\|g\|^2\le CT^{-1}\(\|u(0)\|^2+\|g\|^2\).
$$
This and \eqref{1.dis} prove \eqref{1.smo}
and guarantee that the $\mathcal H^1$-ball
$$
\mathcal B_1:=\{u_0\in\mathcal H^1,\ \ \|u_0\|^2_{\mathcal H^1}\le R\}
$$
is a bounded  absorbing set for the
solution semigroup $S(t)$
if $R$ is large enough.
 Thus, $\mathcal A\subset\mathcal B_1$ is bounded
 in $\mathcal H^1$ and the proof is complete.
\end{proof}
\begin{remark} Arguing analogously, it is not difficult to show that the smoothness of the global attractor $\mathcal A$ is restricted by the smoothness of the external forces $g$. In particular, the attractor will be $C^\infty$-smooth (analytic) if the external forces are $C^\infty$-smooth (analytic).
\end{remark}

\section{Fractal dimension of the attractor: upper bounds}\label{s3}
In this section, we prove that the global attractor constructed
above has  finite fractal (box-counting) dimension and give
explicit upper bounds for this dimension in terms of the physical
parameters $\nu$ and $\alpha$ and the norms of the right-hand
side $g$ in homogeneous Sobolev spaces.  To this end, we first need to remind some
definitions, see \cite{temam,BV} for  more detailed exposition.

\begin{definition}\label{Def2.dim} Let $\mathcal A\subset\mathcal H$ be a compact set in a metric space $\mathcal H$. Then, by the Hausdorff criterion, for every $\varepsilon>0$ it can be covered by the finite number of $\varepsilon$-balls in $\mathcal H$. Let $N_\varepsilon(\mathcal A,\mathcal H)$ be the minimal number of such balls. Then, the fractal (box-counting) dimension of $\mathcal A$ in $\mathcal H$ is defined via the following expression:
$$
\dim_f(\mathcal A,\mathcal H):=\limsup_{\varepsilon\to0}\frac{\ln N_{\varepsilon}(\mathcal A,\mathcal H)}{\ln\frac1\varepsilon}.
$$
The fractal dimension coincides with the usual dimension if the set
$\mathcal A$ is regular enough (for instance, for the case
where $\mathcal A$ is a Lipschitz manifold in $\mathcal H$),
but may be non-integer for irregular sets (for instance,
for the standard ternary Cantor set $K\subset[0,1]$ this
dimension is $\frac{\ln2}{\ln3}$),
see, for instance,  \cite{robinson} for more details.
\end{definition}

We now estimate the dimension of the global attractor $\mathcal{A}$
by means of the so-called volume contraction method~\cite{CF85},
\cite{temam}. The solution semigroup $S_t$ is uniformly quasi-differentiable
on the attractor in the sense that there exists (for a fixed $t\ge0$) a linear bounded operator
$D S(t,u_0)$ such that
\begin{equation}\label{Differ}
\|S(t)u_1-S(t)u_0-D S(t,u_0)\cdot(u_1-u_0)\|\le h(\|u_1-u_0\|),
\end{equation}
where $h(r)/r\to0$ as $r\to0$, and $u_0,u_1\in\mathcal{A}$.

The quasi-differential $D S(t,u_0)$ is the solution operator
$\xi\to v(t)$ of the following equation of variations:
\begin{equation}\label{2.eq-var2}
\partial_t v=L(t,u_0)v:=-\Pi\bigl((v,\Nx)u(t)+(u(t),\Nx)v\bigr)-\alpha v+\nu\Dx v,\ v(0)=\xi,
\end{equation}
where $\Pi:[L^2(\mathbb R^2)]^2\to\mathcal H$ is
the Leray ortho-projection  onto the divergence free vector fields, and $u(t)=S(t)u_0$
is the solution lying on the attractor and parameterized by $u_0\in\mathcal{A}$.
For the proof see~\cite{BV1}, where it is also shown that the
solution semigroup is even  differentiable for all $u_0\in\mathcal{H}$ and
the differential $D S(t,u_0)$ depends continuously on the point $u_0$.

We define for $m=1,2\dots$ the numbers $q(m)$ (the sums of the first
$m$ global Lyapunov exponents)
$$
q(m)=\limsup_{t\to\infty}\ \sup_{u_0\in {\mathcal A}}\ \
\sup_{\{v_j\}_{j=1}^m\in\mathcal H^1}
\frac{1}t
\int_0^t \sum_{j=1}^m\bigl({ L}(\tau,u_0)v_j,v_j\bigr)d\tau,
$$
where the  supremum closest to the integral is taken with respect
to all $L^2$-orthonormal families   $\{v_j\}_{j=1}^m\in\mathcal H^1$.

To define $q(m)$ for all real $m\ge1$ we just linearly
interpolate between $q(m)$ and $q(m+1)$ so that $q(m)$
is now a piece-wise linear continuous function of $m$.

The following theorem is the key technical tool for estimating the dimension of the attractor via the so-called volume contraction method.
\begin{theorem}\label{Th2.vol-contr} Let $S(t):\mathcal H\to\mathcal H$ be the solution semigroup associated with problem \eqref{1.ns-main}
 in a Hilbert space $\mathcal H$ and let $\mathcal A$ be a compact invariant
 set of $S_t$ in $\mathcal H$: $S(T)\mathcal{A}=\mathcal{A}$.

 Suppose  that
the semigroup  $S(t)$ is uniformly quasi-differentiable for every fixed $t$ on $\mathcal A$
in the sense of~\eqref{Differ}.

Suppose further that the quasi-differential $D S(t,u_0)$
depends continuously on the initial point $u_0\in\mathcal{A}$ as
a map $D S(t,\cdot):u_0\to\mathcal{L}(\mathcal{H},\mathcal{H})$.

Suppose that there exists number $m>0$ such that $q(m)<0$. Then
$$
\dim_f\mathcal{A}<m.
$$
\end{theorem}
For the proof of this theorem see \cite{CF85}, \cite{temam} in the case of Hausdorff dimension and \cite{CI} for the fractal dimension.

\begin{remark}
The condition on the continuity of the quasi-differentials with respect to  the initial point
is redundant in the case of the Hausdorff dimension.
It is also redundant in the case of the fractal dimension  if the graph of
$q(m)$ lies below the straight line joining the
points $(m-1,q(m-1))$ and $(m,q(m))$, where   $q(m)<0$ and  $q(m-1)\ge0$, see
\cite{Ch-V-book,CI01}.

Also, in applications to infinite dimensional dissipative dynamical
systems an upper bound for  $q(m)$ is usually found in the
form
$$
q(m)\le -c_1m^{\gamma}+c_2,\quad\gamma\ge1.
$$
For example, as we shall shortly see, $\gamma=1$ in our case. In this case,
also without the continuity condition, we have
$$
\dim_f\mathcal{A}\le (c_2/c_1)^{1/\gamma}.
$$
\end{remark}

To apply this theorem for obtaining the fractal dimension of the attractor $\mathcal A$ of
the semigroup $S(t)$ generated by damped Navier-Stokes equation
\eqref{1.ns-main} we need to state the Lieb-Thirring inequality which plays a fundamental role in estimating the quantities $q(m)$.

\begin{lemma}[{\bf Lieb--Thirring inequality}]\label{T:L-T}
Let
$\{v_j\}_{j=1}^m\in {H}^1(\mathbb{R}^2)^2$
 be a family of orthonormal  vector-functions and let
$\divv v_j=0$.
 Then the following inequality holds for
 $\rho(x)=\sum_{k=1}^m|v_k(x)|^2$:
\begin{equation}\label{L-T}
\|\rho\|^2=
\int_{\mathbb{R}^2}\biggl(\sum_{j=1}^m|v_j(x)|^2\biggr)^2dx \le
c_\mathrm{LT}
\sum_{j=1}^m \| \Nx v_j\|^2,\quad
c_\mathrm{LT}
\le\frac1{2\sqrt{3}}\,.
\end{equation}
\end{lemma}
\begin{proof}
The proof (see~\cite{I-Stokes}) is a reduction to the scalar case
(which works in two dimensions) \cite{CI} and the
use of the main result in ~\cite{D-L-L}.
\end{proof}

 We first set $\nu=1$ and $\alpha=1$ and begin
with estimating the $m$-trace of the operator $L(t,u_0)$
in the equation of variations~\eqref{2.eq-var2} for this particular case. 
The general case will be reduced later to the case $\nu=\alpha=1$ by the proper scaling.

\begin{proposition}\label{Prop1.trace} Let $\nu=1$ and $\alpha=1$. Then, the following estimate holds:
\begin{multline}\label{2.tr-est}
\limsup_{T\to\infty}\sup_{u_0\in\mathcal A}\sup_{\{v_j\}_{j=1}^m\in\mathcal H^1}
\frac1T\int_0^T
\sum_{j=1}^m\bigl({ L}(t,u_0)v_j,v_j\bigr)\,dt\le
\\\le-m +\frac{1}{16\sqrt{3}} \,\limsup_{T\to\infty}\sup_{u_0\in\mathcal A}\frac1T
\int_0^T\|\Nx u(t)\|^2\,dt,
\end{multline}
where $u(t)=S(t)u_0$.
\end{proposition}
\begin{proof} Let $v_1,\dots,v_m\in\mathcal H^1$ be an orthonormal family in $\mathcal H$. 
Then, integrating by parts and using that the vector fields $v_i$ are divergence free, we get
$$
\allowdisplaybreaks
\aligned
&\sum_{j=1}^m(L(u(t))v_j,v_j)=\\&=-\sum_{j=1}^m\|\Nx v_j\|^2-
\sum_{j=1}^m\|v_j\|^2
-\int_{\mathbb{R}^2}\sum_{j=1}^m\sum_{i,k=1}^2v_j^k\partial_{x_k}u^iv_j^i\,dx=\\&=
-m-\sum_{i=1}^m\|\Nx v_i\|^2-\int_{\mathbb{R}^2}\sum_{j=1}^m\sum_{i,k=1}^2v_j^k\partial_{x_k}u^iv_j^i\,dx.
\endaligned
$$
To estimate the right-hand side above, we use the following point-wise inequality
\begin{equation}\label{strange}
\big|\sum_{k,i=1}^2
v^k\partial_{x_k}u^iv^i\big|\le
2^{-1/2}|\Nx u||v|^2
\end{equation}
which holds for any $v=(v^1,v^2)\in\mathbb R^2$ and any Jacobi matrix $\Nx u=\(\partial_{x_i}u^j\)_{i,j=1}^2\in\mathbb R^4$ such that $\partial_{x_1}u^1+\partial_{x_2}u^2=0$, see \cite[Lemma~4.1]{CI}. Indeed,
setting
$$
v=(\xi,\eta),\qquad \Nx u=A:=\left(
                          \begin{array}{cr}
                            a & b \\
                            c & -a \\
                          \end{array}
                        \right),
$$
we have by the Cauchy-Schwartz inequality
$$
\aligned
&(Av,v)=a(\xi^2-\eta^2)+(b+c)\xi\eta=a(\xi^2-\eta^2)+((b+c)/2)\,2\xi\eta\le\\\le
&\sqrt{a^2+(b+c)^2/4}\,\sqrt{(\xi^2-\eta^2)^2+4\xi^2\eta^2}=
\sqrt{a^2+(b+c)^2/4}\cdot |v|^2\le\\\le
&\sqrt{a^2+(b^2+c^2)/2}\cdot|v|^2=\frac1{\sqrt{2}}\sqrt{a^2+b^2+c^2+a^2}\cdot|v|^2,
\endaligned
$$
and estimate \eqref{strange} is proved. Using this pointwise estimate and
the Lieb--Thirring inequality for
divergence free vector fields~\eqref{L-T}, we finally have
$$
\aligned
&\sum_{j=1}^m(L(u(t))v_j,v_j)\le-m-\sum_{i=1}^m\|\Nx v_i\|^2+\frac1{\sqrt2}\int_{\mathbb R^2}\rho(x)|\Nx u(x)|\,dx\le\\
\le& -m-\sum_{i=1}^m\|\Nx v_i\|^2+\frac1{\sqrt{2}}\|\rho\|\|\Nx u(t)\|\le\\\le
&-m-\sum_{i=1}^m\|\Nx v_i\|^2+\frac1{\sqrt{2}}\biggl(c_\mathrm{LT}
\sum_{j=1}^m \| \Nx v_j\|^2\biggr)^{1/2}\|\Nx u(t)\|\le\\\le
&-m+\frac{c_\mathrm{LT}}8\|\Nx u(t)\|^2=
-m+\frac1{16\sqrt{3}}\|\Nx u(t)\|^2.
\endaligned
$$
Integrating this inequality for $t\in[0,T]$ and taking the supremum over all $u_0\in\mathcal A$,
 we obtain \eqref{2.tr-est} and finish the proof of the proposition.
\end{proof}
\begin{theorem}\label{Cor2.dim} Let the assumptions of Theorem \ref{Th1.main} hold
and let $\nu=\alpha=1$. Then the fractal dimension of the global attractor
$\mathcal A$ of problem \eqref{1.ns-main} satisfies the following estimate:
\begin{equation}\label{2.dim-nu}
\dim_f\mathcal A\le\frac{1}{16\sqrt{3}}
\limsup_{T\to\infty}\frac1T\sup_{u_0\in\mathcal A}\int_0^T\|\Nx u(t)\|^2\,dt,
\end{equation}
where $u(t)=S(t)u_0$.
\end{theorem}
\begin{proof}
This estimate is a corollary of Theorem \ref{Th2.vol-contr} and estimate \eqref{2.tr-est}.
\end{proof}
Thus, we only need to estimate the integral in the RHS of \eqref{2.dim-nu}.
We assume below that the right-hand side $g$ belongs to
 the homogeneous Sobolev space $\dot H^s(\mathbb R^2)=(-\Dx)^{-s/2}L^2(\mathbb R^2)$, $s\in\mathbb R$
with norm:
\begin{equation}\label{2.hom}
\|u\|^2_{\dot H^s}:=\int_{\mathbb R^2}|\xi|^{2s}|\hat u(\xi)|^2\,d\xi,
\end{equation}
where $$\hat u(\xi):=\frac1{2\pi}\int_{\mathbb R^2}u(x)e^{- i\xi x}\,dx$$ is the
 Fourier transform of $u$, see \cite{triebel} for more details. Then, obviously,
\begin{equation}\label{2.hom-sp}
\|u\|_{\dot H^0}=\|u\|,\ \ \|u\|_{\dot H^1}=\|\Nx u\|,
\ \ \|u\|_{\dot H^2}=\|\Dx  u\|,
\end{equation}
and the following interpolation inequalities immediately follow from
definition~\eqref{2.hom} and the Cauchy--Schwartz inequality:
\begin{equation}\label{2.hom-int}
\aligned
&\|u\|_{\dot H^s}\le \|u\|^{1-s}\|\Nx u\|^s,\\
&\|\Dx u\|_{\dot H^{-s}}=\|\Nx u\|_{\dot H^{1-s}}\le
\|\Nx u\|^{s}\|\Dx u\|^{1-s},\ \ s\in[0,1].
\endaligned
\end{equation}
\begin{corollary}\label{Cor2.dim-est} Let the assumptions of
Theorem \ref{Th1.main} hold and let, in addition,
$g\in\dot H^{-s}$ for some $s\in[0,1]$ and $\nu=\alpha=1$.
Then the fractal dimension of the attractor $\mathcal A$
of problem \eqref{1.ns-main} satisfies   the following estimate:
\begin{equation}\label{2.dim-attr-s}
\dim_f \mathcal A\le \frac{1-s^2}{64\sqrt{3}}\(\frac{1+s}{1-s}\)^s\|g\|^2_{\dot H^{-s}}.
\end{equation}
\end{corollary}
\begin{proof} From the energy equality \eqref{1.energy} with $\nu=\alpha=1$, interpolation inequalities \eqref{2.hom-int} and the Young inequality, we get
\begin{multline*}
\frac12\frac d{dt}\|u\|^2_{L^2}+\|u\|^2_{L^2}+\|\Nx u\|^2_{L^2}=
(g,u)\le\|g\|_{\dot H^{-s}}\|u\|_{\dot H^s}\le\\\le \|g\|_{\dot H^{-s}}\|\Nx u\|^s_{L^2}\|u\|^{1-s}_{L^2}\le
\frac{\varepsilon^2}2\|g\|^2_{\dot H^{-s}}+\frac{\delta^p}p\|\Nx u\|^2_{L^2}+\frac{\gamma^q}q\|u\|^2_{L^2},
\end{multline*}
where $p=\frac2{s}$, $q=\frac2{1-s}$ and positive numbers $\varepsilon,\delta,\gamma$ are such that $\varepsilon\delta\gamma=1$. Fixing now the parameter $\gamma$ in such way that
$$
\frac{\gamma^q}q=1\ \ \Rightarrow \ \ \gamma=\(\frac{1-s}2\)^{-\frac{1-s}2},
$$
and integrating over $t\in[0,T]$, we arrive at
$$
\frac1{2T}\(\|u(t)\|^2-\|u(0)\|^2\)+\(1-\frac{\delta^p}p\)\frac1T\int_0^T\|\Nx u(t)\|^2
\,dt\le \frac{\varepsilon^2}2\|g\|_{\dot H^{-s}}^2.
$$
From the dissipative estimate \eqref{1.dis} we conclude that
$$
\lim_{T\to\infty}\sup_{u_0\in\mathcal A}\frac1{2T}\(\|u(t)\|^2-\|u(0)\|^2\)=0
$$
and, therefore,
\begin{equation}\label{2.huge}
\limsup_{T\to\infty}\frac1T\sup_{u_0\in\mathcal A}\int_0^T\|\Nx u(t)\|^2
\,dt\le \frac{\varepsilon^2}2\(1-\frac{\delta^p}p\)^{-1}\|g\|_{\dot H^{-s}}^2
\end{equation}
and we only need to optimize  the coefficient in the RHS with respect
to $\varepsilon$ and $\delta$. Indeed, since $\varepsilon\delta\gamma=1$, we conclude that
$$
\varepsilon\delta=\(\frac{1-s}2\)^{\frac{1-s}2}\ \Rightarrow\ \varepsilon=\delta^{-1}\(\frac{1-s}2\)^{\frac{1-s}2}
$$
and
$$
\frac{\varepsilon^2}2\(1-\frac{\delta^p}p\)^{-1}=
\frac12\(\frac{1-s}2\)^{1-s}\frac1{x\left(1-\frac s2 x^{\frac1s}\right)},
$$
where $x:=\delta^2$. Thus, it only remains to maximize the function
$$
f(x):=x\left(1-\frac s2 x^{\frac1s}\right)
$$
on the interval $x\ge0$:
$$
f'(x)=1-\frac{s+1}2x^{\frac1s}=0  \Rightarrow
x=\(\frac2{s+1}\)^s \Rightarrow f(x)=\frac1{s+1}\(\frac2{s+1}\)^s.
$$
Inserting the obtained estimates into the right-hand side of \eqref{2.huge}, we finally get

$$
\limsup_{T\to\infty}\frac1T\sup_{u_0\in\mathcal A}\int_0^T\|\Nx u(t)\|^2\,dt\le \frac{1-s^2}4\(\frac{1+s}{1-s}\)^s\|g\|_{\dot H^{-s}}^2.
$$
This estimate together with \eqref{2.dim-nu} completes the proof of the corollary.
\end{proof}
The next corollary gives the analogous result for  $g\in\dot H^s$ with  $s\ge0$.
\begin{corollary}\label{Cor2.dim-est1} Let the assumptions of
Theorem~\ref{Th1.main} hold and let, in addition, $g\in\dot H^{s}$
for some $s\in[0,1]$ and $\nu=\alpha=1$. Then the fractal dimension
of the attractor $\mathcal A$ of problem \eqref{1.ns-main} satisfies
 the following estimate:
\begin{equation}\label{2.dim-attr-s1}
\dim_f\mathcal A\le \frac{1-s^2}{64\sqrt{3}}\(\frac{1+s}{1-s}\)^s\|g\|^2_{\dot H^{s}}.
\end{equation}
\end{corollary}
\begin{proof} From the second energy equality \eqref{1.strong-en} with $\nu=\alpha=1$, interpolation inequalities \eqref{2.hom-int} and the Young inequality, we get
\begin{multline*}
\frac12\frac d{dt}\|\Nx u\|^2+\|\Nx u\|^2+\|\Dx u\|^2=
(g,\Dx u)\le\|g\|_{\dot H^{s}}\|\Dx u\|_{\dot H^{-s}}\le\\\le \|g\|_{\dot H^{s}}
\|\Nx u\|^s\|\Dx u\|^{1-s}\le\\\le
\frac{\varepsilon^2}2\|g\|^2_{\dot H^{s}}+\frac{\delta^p}p\|\Nx u\|^2+\frac{\gamma^q}q\|\Dx u\|^2,
\end{multline*}
where the exponents $p,q$ and the constants $\varepsilon,\delta,\gamma$ are {\it exactly} the same as in the proof of the previous corollary.
Thus, using the strong dissipative estimate \eqref{1.strong-dis} and arguing exactly as in the proof of the previous corollary, we end up with
$$
\limsup_{T\to\infty}\frac1T\sup_{u_0\in\mathcal A}\int_0^T\|\Nx u(t)\|^2
\,dt\le \frac{1-s^2}4\(\frac{1+s}{1-s}\)^s\|g\|_{\dot H^{s}}^2
$$
which together with \eqref{2.dim-nu} finishes the proof of the corollary.
\end{proof}
Thus, combining estimates \eqref{2.dim-attr-s} and \eqref{2.dim-attr-s1}, we end up with
\begin{equation}\label{2.dim-attr-s2}
\dim_f\mathcal A\le\frac{1-s^2}{64\sqrt{3}}\(\frac{1+|s|}{1-|s|}\)^{|s|}\|g\|^2_{\dot H^{s}}
\end{equation}
which holds for the case $\nu=\alpha=1$ if $g\in\dot H^{-1}\cap \dot H^1$ and $s\in[-1,1]$.
\par
Finally, we need the analogue of estimate \eqref{2.dim-attr-s2} for general $\nu,\alpha>0$.
We reduce this general case to the particular case
 $\nu=\alpha=1$ by the proper scaling of $t$, $x$ and $u$.
 Indeed, let $u=u(t,x)$ be a solution of \eqref{1.ns-main}
 with arbitrary $\nu,\alpha>0$. Then, taking
$$
t':=\alpha t,\ \ x'=\(\frac\alpha\nu\)^{1/2}x,\ \ \tilde u=\frac1{(\alpha\nu)^{1/2}}u
$$
we see that the function $\tilde u(t',x'):=\frac1{(\alpha\nu)^{1/2}}u(t,x)$
solves equation \eqref{1.ns-main} with $\nu=\alpha=1$ and the
external forces $\tilde g(x')=\alpha^{-1}(\alpha\nu)^{-1/2}g(x)$.
Since the fractal dimension of the attractor does not change under
this scaling, using the obvious scaling properties of the $\dot H^s$-norm:
$$
\|g\(\gamma\cdot\)\|^2_{\dot H^s}=\gamma^{2(s-1)}\|g\|^2_{\dot H^s},
$$
we have proved the following result.

\begin{theorem} Let the assumptions of Theorem \ref{Th1.main}
hold (now with arbitrary positive $\nu$ and $\alpha$) and let,
in addition $g\in \dot H^s$ for some $s\in[-1,1]$. Then, the
fractal dimension of the attractor $\mathcal A$ in $\mathcal H$ satisfies the following estimate:
\begin{equation}\label{2.main-est}
\dim_f\mathcal A\le \frac{1-s^2}{64\sqrt{3}}
\(\frac{1+|s|}{1-|s|}\)^{|s|}\frac1{\alpha^2\nu^2}\(\frac\nu\alpha\)^s\|g\|^2_{\dot H^{s}}.
\end{equation}
\end{theorem}
\begin{proof} Indeed, the above estimate follows from \eqref{2.dim-attr-s2} and the identity
$$
\|\tilde g\|^2_{\dot H^s}=\frac1{\alpha^2\nu^2}\(\frac\nu\alpha\)^s\|g\|^2_{\dot H^s}.
$$
\end{proof}

If $g\in \dot H^{-1}\cap\dot H^1$, then the rate of growth of the
estimate~\eqref{2.main-est} with respect to $\nu$ as $\nu\to0$ is
the smallest when $s=1$. In this case we have

\begin{corollary}\label{Cor:s=1} Suppose that
$g\in \dot H^1$. Then the
fractal dimension of the attractor $\mathcal A$ satisfies
\begin{equation}\label{2.main-est-s=1}
\dim_f\mathcal A\le \frac{1}{16\sqrt{3}}
\frac{\|\rot g\|^2}{\alpha^3\nu}.
\end{equation}
\end{corollary}
\begin{proof}
Since  $\divv g=0$, it follows that
$$
\|g\|^2_{\dot H^1}=\|\Nx g\|^2=\|\divv g\|^2+\|\rot g\|^2=
\|\rot g\|^2.
$$
\end{proof}

\subsubsection*{\bf{Acknowledgements}}\label{SS:Acknow}
This work was done with the financial support
from the Russian Science Foundation (grant no. 14-21-00025).

\end{document}